\documentclass[12pt,a4paper,psamsfonts]{amsart}

\usepackage{fancyhdr}
\usepackage{appendix}
\usepackage{amssymb,amscd,amsxtra,calc}
\usepackage{mathrsfs}
\usepackage{cmmib57}
\usepackage{multirow}
\usepackage[all]{xy}
\usepackage{longtable}
\usepackage[colorlinks,linkcolor=blue,anchorcolor=blue,citecolor=green]{hyperref}

\usepackage{mathabx,epsfig}
\def\acts{\ \rotatebox[origin=c]{-90}{$\circlearrowright$}\ }
\def\racts{\ \rotatebox[origin=c]{90}{$\circlearrowleft$}\ }

\theoremstyle{plain}
    \newtheorem{thm}{Theorem}[section]

     \newtheorem{conjecture}[thm]{Conjecture}
    
    \newtheorem{example}[thm]{Example}
    
    \newtheorem{proposition}[thm]{Proposition}
    \newtheorem{question}[thm]{Question}
    \newtheorem{theorem}[thm]{Theorem}

\theoremstyle{definition}
    \newtheorem{definition}[thm]{Definition}
    
    \newtheorem*{notation*}{Notation and Terminology}
    \newtheorem{remark}[thm]{Remark}
\theoremstyle{remark}

    \newtheorem{setup}[thm]{}

\newcommand{\C}{\mathbb{C}}

\newcommand{\Q}{\mathbb{Q}}

\newcommand{\R}{\mathbb{R}}

\newcommand{\Z}{\mathbb{Z}}

\newcommand{\alb}{\operatorname{alb}}

\newcommand{\Gal}{\operatorname{Gal}}

\newcommand{\id}{\operatorname{id}}

\newcommand{\Ker}{\operatorname{Ker}}
\newcommand{\NE}{\overline{\operatorname{NE}}}
\newcommand{\Nef}{\operatorname{Nef}}

\newcommand{\NS}{\operatorname{NS}}

\newcommand{\PE}{\operatorname{PE}}

\newcommand{\Per}{\operatorname{Per}}

\newcommand{\SEnd}{\operatorname{SEnd}}

\newcommand{\Supp}{\operatorname{Supp}}

\newcommand{\variety}{\operatorname{variety}}

\newcommand{\N}{\operatorname{N}}

\newcommand{\Pic}{\operatorname{Pic}}

\newcommand{\bk}{\mathbf{k}}

\begin{document}

\title[Equivariant minimal model program]
{Advances in the equivariant minimal model program and their applications in complex and arithmetic dynamics}

\author{Sheng Meng, De-Qi Zhang}

%Sheng Meng
\address{
    \textsc{School of Mathematical Sciences, Ministry of Education Key Laboratory of Mathematics and Engineering Applications \& 
    Shanghai Key Laboratory of PMMP}\endgraf
    \textsc{East China Normal University, Shanghai 200241, China}\endgraf
}
\email{smeng@math.ecnu.edu.cn}

\address
{
\textsc{Department of Mathematics} \endgraf
\textsc{National University of Singapore,
Singapore 119076, Republic of Singapore
}}
\email{matzdq@nus.edu.sg}

\begin{abstract}
This note reports some advances in the Equivariant Minimal Model Program (EMMP) for non-isomorphic surjective endomorphisms
and their applications in complex and arithmetic dynamics.
\end{abstract}



\maketitle
\tableofcontents

\section{Introduction}

We work over an algebraically closed field of characteristic $0$ unless otherwise specified.
A fair amount is known in any characteristic and we refer to \cite{CMZ20} and its references for more information.

Consider a finite sequence of dominant rational maps of projective varieties
$$(*): \hskip 1pc X_1\dasharrow X_2\dasharrow\cdots\dasharrow X_r.$$
Let $f:X_1\to X_1$ be a surjective endomorphism.
We say the sequence $(*)$ is {\it $f$-equivariant} if there exists the following commutative  diagram 
$$\xymatrix{
X_1\ar@{-->}[r]\ar[d]^{f_1} &X_2\ar@{-->}[r]\ar[d]^{f_2} &\cdots\ar@{-->}[r] &X_r\ar[d]^{f_r}\\
X_1\ar@{-->}[r] &X_2\ar@{-->}[r] &\cdots\ar@{-->}[r] &X_r\\
}
$$
where $f_1=f$ and all the $f_i$ are surjective endomorphisms.
%Denote by $\SEnd(X)$ the monoid of all surjective endomorphisms of $X$.
Let $S$ be a set of surjective endomorphisms of $X$.
%subset of $\SEnd(X)$.
We say the sequence $(*)$ is {\it $S$-equivariant} if it is $g$-equivariant for every $g\in S$.

The first part of this paper focuses on the Equivariancy of the Minimal Model Program (MMP), i.e.,
to construct an Equivariant Minimal Model Program (EMMP).
Late on, applications will be given to several questions and conjectures in complex and arithmetic dynamics.
We will outline the ideas while referring to the original papers for detailed proofs.

\par \vskip 1pc
{\bf Acknowledgement.}
The authors would like to thank Jason Bell, Yohsuke Matsuzawa, Matthew Satriano, Joe Silverman, Burt Totaro, Long Wang and Junyi Xie for the valuable comments on this note.
The first author is supported by Science and Technology Commission of Shanghai Municipality (No. 22DZ2229014) 
and a National Natural Science Fund, and the second author
is supported by ARF: A-8000020-00-00 of NUS. 
Many thanks to CIRM Luminy, IMS at NUS, MFO at Obwerwolfach, NCTS at Taipei, and Simons Foundation for the kind support.

\section{The Minimal Model Program (MMP)}

In this section, we briefly introduce the classical Minimal Model Program (MMP).
We begin with some terminology.

\begin{setup}
{\rm

Let $X$ be a projective variety of dimension $n$.
Denote by $\Pic(X)$ the group of Cartier divisors modulo linear equivalence and $\Pic^0(X)$ the subgroup of the classes in $\Pic(X)$ which are algebraically equivalent to $0$.
Let $\NS(X)=\Pic(X)/\Pic^0(X)$ be the {\it N\'eron-Severi group}.
Denote by $\N^1(X):=\NS(X) \otimes_{\Z} \mathbb{R}$.
Let $Z$ and $Z'$ be two $r$-cycles with real coefficients.
$Z$ is said to be {\it numerically equivalent} to $Z'$ (denoted by $Z\equiv Z'$) if $(Z-Z')\cdot H_1\cdots H_r=0$ for any Cartier divisors $H_1,\cdots, H_r$;
it was called weakly numerically equivalent in our earlier papers and is weaker than the numerical equivalence defined in Fulton's book (cf.~\cite[Definition 19.1]{Ful98}).
For Cartier divisors (regarded as $(\dim X -1)$-cycles), the numerical equivalence here coincides with the usual one; see \cite[Lemma 3.2]{Zha16}.
Denote by $\N_r(X)$ the quotient vector space of $r$-cycles with real coefficients modulo the numerical equivalence. 

%In some of the historical notes, the equivalence is called the weak numerical equivalence for some other purpose which we will not meet in this paper. 

We refer to \cite[Chapter II]{KM98} and \cite{Fuj15} for some basics of singularities of a variety or a pair, like, {\it lc (log canonical), klt (Kawamata log terminal), 
canonical, or terminal singularity}.
A pair $(X, \Delta)$ is a normal variety $X$ with a $\Q$-Cartier (boundary) divisor $\Delta$ such that all coefficients of $\Delta$ are in $[0, 1] \cap \Q$ and
the adjoint divisor $K_X + \Delta$ is $\Q$-Cartier.

Let $(X, \Delta)$ be a pair with only lc singularities.
Let $\pi:X \to S$ be a projective morphism to a variety $S$.
We use $\N_1(X/S)$ to denote the space of $1$-cycles generated by curves in fibres of $\pi$, modulo numerical equivalence.
$\NE(X/S)$ denotes the closure of the effective $1$-cycles in $\N_1(X/S)$.
The classical {\it cone theorem} asserts that
$$\NE(X/S)=\NE(X/S)_{(K_X+\Delta) \ge 0}+\sum_i \mathbb{R}_{\ge 0}[C_i]$$
where each $\mathbb{R}_{\ge 0}[C_i]$ is an {\it extremal ray} of $\NE(X/S)$, generated by a rational curve $C_i$ 
with $(K_X+\Delta)\cdot C_i<0$.
%($C_i$ can be chosen to be rational if the pair is further klt).
Moreover, for each $C_i$, the base point free theorem ensures a contraction morphism, with connected fibres, 
$\tau: X \to Y$ over $S$ to a normal projective variety $Y$, called 
{\it the contraction of the $(K_X + \Delta)$-negative extremal ray} $\mathbb{R}_{\ge 0}[C_i]$,
such that exactly the curves in the class of $\mathbb{R}_{\ge 0}[C_i]$ are contracted to points and one obtains the exact sequence
$$0\to \N^1(Y)\xrightarrow[]{\tau^*} \N^1(X)\xrightarrow[]{\cdot C_i} \mathbb{R}\to 0.$$
In particular, we have the relation ($\dim_{\R} \N^1(X) =$) $ \rho(X)=\rho(Y)+1$ between the {\it Picard numbers}.

\par \vskip 1pc
Assume further that $(X, \Delta)$ is a $\Q$-factorial log canonical pair.
There are three cases for the contraction $\tau: X \to Y$ of a $(K_X + \Delta)$-negative extremal ray $\mathbb{R}_{\ge 0}[C]$.

\par \vskip 1pc \noindent
\textbf{Case 1 (Divisorial contraction).}
$\tau$ is a birational contraction with the exceptional locus (of points of $X$ along which $\tau$ is not an isomorphism) being a prime divisor.
Note that $Y$ is still $\Q$-factorial with log canonical singularities.

\par \vskip 1pc \noindent
\textbf{Case 2 (Flipping contraction).} 
$\tau$ is small, i.e., a birational contraction with the exceptional locus being of codimension $\ge 2$ in $X$.
In this case, $Y$ is not $\Q$-factorial and we need to take the flipping diagram:
$$\xymatrix{
X\ar[rd]_{\tau}\ar@{-->}[rr]^\phi&&X^+\ar[ld]^{\tau^+}\\
&Y
}
$$
where $\tau^+$ is a small projective birational morphism and 
\begin{enumerate}
\item $(X^+,\Delta^+)$ is a $\Q$-factorial log canonical pair with $\Delta^+=\phi_*\Delta$,
\item $K_{X^+}+\Delta^+$ is $\tau^+$-ample, and
\item $\rho(X^+)=\rho(Y)+1$.
\end{enumerate}
We refer to \cite{Bir12, HX13} for the existence of log canonical flips.

\par \vskip 1pc \noindent
\textbf{Case 3 (Fano contraction).} 
This is the case where $\dim Y<\dim X$ (and general fibres are of Fano type).
Note that $Y$ is again $\Q$-factorial with log canonical singularities, see \cite{Fuj15}.
}
\end{setup}

\begin{remark}
We refer to \cite{Fuj17} for the non-$\Q$-factorial Minimal Model Program.
\end{remark}

\begin{remark}
There are two very hard conjectures in the Minimal Model Program.
One is the termination of flips asserting that one can not have infinitely many flips when running a Minimal Model Program.
Another one is the abundance conjecture asserting that $K_X$ is semi-ample when it is nef.
In our Equivariant Minimal Model Program, we will overcome these two obstacles with the help of dynamics.
\end{remark}

We end this section by citing two useful results on the two conjectures.
Hashizume and Hu \cite{HH20} proved the special termination of directed flips for non-pseudo-effective pairs.
This is the only termination needed in our EMMP.

\begin{theorem}
Let $(X,\Delta)$ be a $\Q$-factorial log canonical pair over $S$ with $K_X+\Delta$ not being pseudo-effective over $S$.
Then there is a finite sequence of $(K_X+\Delta)$-MMP over $S$
$$(X,\Delta)=(X_2,\Delta_1)\dashrightarrow (X_2,\Delta_2)\dashrightarrow\cdots\dashrightarrow (X_n,\Delta_n)\to Y$$
where $X_i\dashrightarrow X_{i+1}$ is birational  and $(X_n, \Delta) \to Y$ is a Fano contraction.
\end{theorem}

Gongyo \cite{Gon13} proved the abundance conjecture for the numerically trivial log canonical pair.
\begin{theorem}\label{thm-gongyo}
Let $(X,\Delta)$ be a log canonical pair with $K_X+\Delta\equiv 0$.
Then $K_X+\Delta\sim_{\Q} 0$.
\end{theorem}

\section{Various types of endomorphisms and fundamental properties}

We begin with the following (cf. \cite{KR17, Men20, Men23, MZ20b}):

\begin{definition}  Let $f:X\to X$ be a surjective endomorphism of a projective variety $X$.
\begin{enumerate}
\item $f$ is {\it $q$-polarized}  if $f^{\ast}L \sim qL$ for some ample Cartier divisor $L$ and integer $q>1$.
\item $f$ is {\it amplified} (in the sense of Krieger and Reschke \cite{KR17}, which is also the assumption of Theorem \ref{FakThm} in Fakhruddin \cite{Fak03}) if $f^*L-L=H$ for some Cartier divisor $L$ and ample 
%Cartier
divisor $H$.
\item $f$ is {\it int-amplified} if $f^*L-L=H$ for some ample Cartier divisors $L$ and $H$, i.e., the operator $(f^* {-} \id)$ maps an {\bf int}erior point of the nef cone $\Nef(X)$
to another {\bf int}erior point of the same cone and is like an {\bf int}ernal mapping in the sense of Theorem \ref{main-thm-cri} (4), whence the naming of this notion.
\item $f$ is {\it quasi-amplified} if $f^*L-L=H$ for some Cartier divisor $L$ and big Cartier divisor $H$.
\item $f$ is {\it PCD} (a short form for Periodic  Countable Dense) if the set $\Per(f_K)$ of periodic points is countable and Zariski dense in $X_K$ for some uncountable algebraically closed field extension $K/k$.
\item $f$ is of {\it positive entropy} if the spectral radius of $f^*|_{\N^1(X)}$ on the Neron-Severi real space $\N^1(X)=\NS(X)\otimes_{\Z}\R$ is greater than $1$ (see e.g. \cite{DS04}).
\item $f$ is of {\it null entropy} if $f$ is not of positive entropy.
\item $f$ is {\it $p$-cohomologically hyperbolic} if $\delta_p(f)>\delta_i(f)$ for any $i\in \{1,2,\cdots,\dim X\}\backslash\{p\}$, where $\delta_i(f)= \lim\limits_{n \to +\infty} ((f^n)^*H^i \cdot H^{\dim (X) -i})^{1/n}$ for some (and hence any) nef and big divisor $H$ (cf.~\cite{Dan20}).
\end{enumerate}
\end{definition}

A criterion for polarized endomorphisms is given by \cite[Propositions 1.1 and 2.9]{MZ18}.

\begin{proposition}
Let $f:X\to X$ be a surjective endomorphism of a projective variety $X$ and $q>0$ a rational number.
Assume one of the following conditions.
\begin{enumerate}
\item  $f^\ast H \equiv qH$
for some big $\mathbb{R}$-Cartier divisor $H$.
\item  $X$ is normal and $f^\ast H\equiv qH$ for some big Weil $\mathbb{R}$-divisor $H$.
\item  $f^*|_{\NS_{\C}(X)}$ is diagonalizable with all the eigenvalues being of modulus $q$.
\end{enumerate}
Then $q$ is an integer and $f^\ast A\equiv qA$ for some ample Cartier
divisor $A$.
Further, if $q>1$, then $f$ is polarized.
\end{proposition}

A criterion for int-amplified endomorphisms is given in \cite[Theorem 1.1]{Men20}.
We refer to \cite[Theorem 1.1]{Zho21} for a similar criterion in the K\"ahler setting.

\begin{theorem}\label{main-thm-cri}  (see also Proposition \ref{Mats_Prop})
Let $f:X\to X$ be a surjective endomorphism of a projective variety $X$.
Then the following are equivalent.
\begin{itemize}
\item[(1)] The endomorphism $f$ is int-amplified.
\item[(2)] All the eigenvalues of $\varphi:=f^*|_{\N^1(X)}$ are of modulus greater than $1$.
\item[(3)] There exists some big $\R$-Cartier divisor $B$ such that $f^*B-B$ is big.
\item[(4)] If $C$ is a $\varphi$-invariant convex cone in $\N^1(X)$,
then $\emptyset\neq(\varphi-\id_{\N^1(X)})^{-1}(C)\subseteq C$.
\end{itemize}
\end{theorem}

We have the following easy application.

\begin{proposition}
Consider the equivariant dynamical systems
$$\xymatrix{
f \acts X \ar@{-->}[r]^{\pi} &Y\racts g
}$$
where $\pi$ is dominant.
Then the following hold.
\begin{enumerate}
\item If $f$ is $q$-polarized (resp.~int-amplified), then so is $g$.
\item Suppose $\dim X=\dim Y$.
Then $f$ is $q$-polarized (resp.~int-amplified) if so is $g$.
\end{enumerate}
\end{proposition}

A criterion for amplified endomorphisms is given in \cite[Proposition 3.1]{Men23}.

\begin{proposition}\label{prop-a-cri}
Let $f:X\to X$ be a surjective endomorphism of a projective variety $X$.
Then the following are equivalent.
\begin{enumerate}
\item $f$ is amplified.
\item For any $Z\in \NE(X)$, $f_*Z\equiv Z$ implies $Z\equiv 0$. 
\end{enumerate}
\end{proposition}

We give a proof to the well-known result below since no reference is available.
This result is applicable to varieties of Fano type or Mori Dream Spaces.

\begin{proposition}\label{prop-iamp-amp}
Let $f:X\to X$ be an amplified (or quasi-amplified) endomorphism of a projective variety $X$.
Suppose either the nef cone $\Nef(X)$ or the closed cone $\PE(X)$ of pseudo-effective divisors is a rational polyhedron.
Then $f$ is int-amplified.
\end{proposition}

\begin{proof}
We give a proof for $\PE(X)$.
For $\Nef(X)$, the proof is similar.
Since $\PE(X)$ is a rational polyhedron and spans $\N^1(X)$,  our $f^*$ fixes all the (finitely many) extremal rays of $\PE(X)$ 
after iteration and hence $f^*|_{\N^1(X)}$ is a diagonal action.
Let $v_1,\cdots, v_n$ be a basis of $\NS_{\Q}(X)$ such that $\mathbb{R}_{\ge 0} v_i$ are all extremal rays of $\PE(X)$.
Write $f^*v_i=r_i v_i$ with $r_i>0$.
Note that $r_i$ is also an algebraic integer since it is an eigenvalue of $f^*|_{\NS(X)/{\rm (torsion)}}$.
So $r_i$ is a positive integer.
Assume that $r_i=1$ for $i\le t$ and $r_i>1$ for $i>t$.
By the assumption, $f^*D-D=A$ is big for some divisor $D$.
Write $D=\sum\limits_{i=1}^n a_i v_i$ in $\N^1(X)$.
Note that $A=\sum\limits_{i=t+1}^n a_i(r_i-1)v_i$ is big.
Hence $A' := \sum\limits_{i=t+1}^n |a_i|(r_i-1)v_i$ is big too, since $v_i$ is pseudo-effective.
Now $f^*A' - A' = \sum\limits_{i=t+1}^n |a_i|(r_i-1)^2v_i$ is big too.
Thus $f$ is int-amplified by Theorem \ref{main-thm-cri}.
%By the assumption, $f^*D-D=A$ is ample for some divisor $D$. Write $D=\sum\limits_{i=1}^n a_i v_i$ in $\N^1(X)$. Note that $A=\sum\limits_{i=t+1}^n a_i(r_i-1)v_i$ is ample. Let $H=\sum \limits_{i=t+1}^n \max\{a_i,0\} v_i$. Then $$B:=f^*H-H=\sum \limits_{i=t+1}^n \max\{a_i,0\}(r_i-1) v_i =A-\sum \limits_{i=t+1}^n \min\{a_i,0\}(r_i-1) v_i$$  is big since $v_i$ is pseudo-effective. Note that $tH-B$ is pseudo-effective when $t\gg 1$. So $H$ is big. Thus $f$ is int-amplified by Theorem \ref{main-thm-cri}.
\end{proof}

Thanks to Matsuzawa for bringing the following nice equivalence to our attention.

\begin{proposition}\label{Mats_Prop}
Let $f:X\to X$ be a surjective endomorphism of a projective variety $X$ of dimension $n$.
Then $f$ is int-amplified if and only if $f$ is $n$-cohomologically hyperbolic.
\end{proposition}
\begin{proof}
By the log concavity of $\delta_k(f)$ (cf.~\cite[Theorem 1.2]{Gue05}), $f$ is $n$-cohomologically hyperbolic if and only if $\delta_{n-1}(f)<\delta_n(f)=\deg f$.
Note that $\delta_{n-1}(f)$ is the spectral radius of $f^*|_{\N_{1}(X)}$.
By the projection formula, $f_*f^*=(\deg f)\id$ on $\N_1(X)$ and $f^*D\cdot C=D\cdot f_*C$ for any $D\in \N^1(X)$ and $C\in \N_1(X)$.
So $\delta_{n-1}(f)<\deg f$ is equivalent to that all the eigenvalues of $f^*|_{\N^1(X)}$ have modulus $>1$.
Then we apply Theorem \ref{main-thm-cri}.
\end{proof}

The following results \cite[Theorems 1.5 and 1.7]{Men23} give characterization and comparison of amplified and 
PCD endomorphisms of abelian varieties in terms of their actions on the cohomology groups and divisor classes.

\begin{theorem}
Let $f:A\to A$ be a surjective endomorphism of an abelian variety $A$.
Then the following hold.
\begin{itemize}
\item[(1)] $f$ is amplified if and only if no eigenvalue of $f^*|_{H^1(A,\mathcal{O}_A)}$ is of modulus $1$.
\item[(2)] $f$ is PCD if and only if no eigenvalue of $f^*|_{H^1(A,\mathcal{O}_A)}$ is a root of unity.
\end{itemize}
\end{theorem}

\begin{theorem}
Let $f:A\to A$ be a surjective endomorphism of an abelian variety $A$.
Then the following hold.
\begin{itemize}
\item[(1)] $f$ is amplified if and only if $f^*D\not\equiv D$ for any nef $\R$-Cartier divisor $D\not\equiv 0$.
\item[(2)] $f$ is PCD if and only if $f^*D\not\equiv D$ for any nef Cartier divisor $D\not\equiv 0$.
\end{itemize}
\end{theorem}

\begin{remark}
$ $
\begin{enumerate}
\item It is possible that all the eigenvalues of $f^*|_{\N^1(X)}$ are $q$, but $f$ is not polarized, or equivalently,
$f^*|_{\N^1(X)}$ is not diagonalizable, see \cite[Example 10.1]{Men20}.
\item Though polarized endomorphisms are always int-amplified, it is possible for a projective variety to admit an int-amplified endomorphism 
but no polarized endomorphism, see the example constructed in \cite[Section 7]{MY21}.
\item Though the product of polarized endomorphisms is int-amplified, 
it is possible that an int-amplified endomorphism may never split as or even be dominated 
by a product of polarized endomorphisms, see \cite[Example 10.3]{Men20}.
We also refer to \cite{MZg20} for a splitting result.
\item Some K3 surface admits an automorphism of positive entropy that is not PCD,
and some K3 surface admits a PCD automorphism that is not amplified, see \cite[Examples 5.7 and 5.8]{Men23}.
\item We refer to \cite{DS04} for dynamical degrees.
\end{enumerate}

\end{remark}

We end this section with the relation between int-amplified endomorphism and Bott vanishing, due to Kawakami and Totaro \cite[Theorem C]{KT23}.
This way, they gave a different proof of Fano threefolds being toric when admitting int-amplified endomorphisms, see Remark \ref{rmk-fano}.
\begin{theorem}\label{thm-KT}
Let $X$ be a normal projective variety over a perfect field $k$. 
Suppose that $X$ admits an int-amplified endomorphism whose degree is invertible in $k$. 
Then $X$ satisfies Bott vanishing for ample Weil divisors. 
That is,
$$H^i(X, \Omega_X^{[j]}(A))=0$$
for every $i > 0, j\ge 0$, and any ample ($\Q$-Cartier) Weil divisor $A$.
\end{theorem}

\section{Int-amplified Equivariant Minimal Model Program (EMMP)}
In this section, we develop the Equivariant Minimal Model Program (EMMP) for arbitrary surjective endomorphisms of a $\Q$-factorial klt projective variety $X$ 
which admits an int-amplified endomorphism. This way, one is reduced to the study of endomorphisms on 
$Q$-abelian varieties (quasi-\'etale quotient of some abelian variety) and Fano varieties.

Let $\SEnd(X)$ be the set of {\it all surjective endomorphisms} on $X$.
A submonoid $G$ of a monoid $\Gamma$ is said to be of {\it finite-index} in $\Gamma$
if there is a chain $G = G_0 \le G_1 \le \cdots \le G_r = \Gamma$ of submonoids and homomorphisms $\rho_i : G_i \to F_i$ 
such that $\Ker(\rho_i) = G_{i-1}$ and all the $F_i$ are finite
{\it groups}.

An lc pair $(X, \Delta)$ may have infinitely many $(K_X + \Delta)$-negative extremal rays.
However, this case will never happen if $X$ further admits an int-amplified endomorphism;
moreover, every MMP starting from $X$ is equivariant relative to $\SEnd(X)$, up to finite-index (cf. \cite[Theorem 1.1]{MZ20a}).

\begin{theorem}
Let $X$ be a normal projective variety admitting an int-amplified endomorphism. Then we have:
\begin{itemize}
\item[(1)]
Suppose $(X, \Delta)$ is lc for some effective $\R$-divisor $\Delta$. Then there are only finitely many $(K_X + \Delta)$-negative extremal rays.
\item[(2)]
Suppose $X$ is $\Q$-factorial. Then any finite sequence of MMP starting from $X$ is $G$-equivariant for some finite-index submonoid $G$ of $\SEnd(X)$.
\end{itemize}
\end{theorem}

The detailed EMMP \cite[Theorem 1.2]{MZ20a} can be further described.

\begin{theorem}\label{main-thm-GMMP}
Let $f:X\to X$ be an int-amplified endomorphism of a $\mathbb{Q}$-factorial klt projective variety $X$.
Then there exist a finite-index submonoid $G$ of $\SEnd(X)$,
a $Q$-abelian variety $Y$, and a $G$-EMMP over $Y$
$$X=X_0 \dashrightarrow \cdots \dashrightarrow X_i \dashrightarrow \cdots \dashrightarrow X_r=Y$$
(i.e. $\forall$ $g \in G = G_0$ descends to $g_i \in G_i$ on each $X_i$), with every $X_i \dashrightarrow X_{i+1}$
a divisorial contraction, a flip or a Fano contraction, of a $K_{X_i}$-negative extremal ray, such that:
\begin{itemize}
\item[(1)]
There is a finite quasi-\'etale Galois cover $A \to Y$ from an abelian variety $A$
such that $G_Y := G_r$ lifts to a submonoid $G_A$ of $\textup{SEnd}(A) \le \textup{End}_{\variety}(A)$.
\item[(2)]
If $g$ in $G$ is polarized (resp. int-amplified) then so are
its descending $g_i$ on $X_i$ and the lifting to $A$ of $g_r$ on $X_r = Y$.
\item[(3)]
If $g$ in $G$ is amplified and
its descending $g_r$ on $X_r$ is int-amplified, then $g$ is also int-amplified.
\item[(4)]
If $K_X$ is pseudo-effective, then $X=Y$ and it is $Q$-abelian.
\item[(5)]
If $K_X$ is not pseudo-effective, then for each $i$, $X_i\to Y$ is equi-dimensional holomorphic with every fibre (irreducible) 
rationally connected (in the sense of \cite[IV, Definition 3.2]{Kol96}).
%and $f_i$ is int-amplified.
The $X_{r-1}\to X_r = Y$ is a Fano contraction.
\item[(6)]
For any subset $H \subseteq G$ and its descending $H_Y \le \SEnd(Y)$, and any $K = \Q$ or $K = \C$,
$H$ acts on $\NS_{K}(X)$ as commutative diagonal matrices with respect
to a suitable basis if and only if so does $H_Y$.
\end{itemize}
\end{theorem}

Yoshikawa \cite[Theorem 1.3]{Yos21} recently constructed, in a sophisticated way, a new sequence of Equivariant log MMP
(called ``maximal sequence  of steps of MMP of canonical bundle formula type''), together with a new terminology ``$f$-pair'', 
providing a new perspective on the ramification divisor.
This way, he showed the following:

\begin{theorem}\label{thm-yoshikawa}
Let $f:X\to X$ be an int-amplified endomorphism of a $\Q$-factorial klt projective variety. 
Then there exists an $f$-equivariant quasi-\'etale cover $\mu:\widetilde{X}\to X$ such that 
the albanese morphism $\alb:\widetilde{X}\to A$ is a fiber space and $\widetilde{X}$ is of Fano type over $A$, 
i.e., $(X,\Delta)$ is klt and $-(K_X+\Delta)$ is ample over $A$ for some effective $\Q$-divisor $\Delta$.
\end{theorem}

\begin{remark}
We refer to \cite{NZ10, Zha10, MZ18, Men20, MZ20a} for historical notes.
We also refer to \cite{CMZ20, MZ20b, Zho21} for int-amplified EMMP in the setting of positive characteristics or K\"ahler spaces.

\end{remark}

\section{Non-isomorphic EMMP for projective surfaces}

Nakayama \cite{Nak02} first showed the finiteness of negative curves on a smooth projective surface 
admitting a non-isomorphic surjective endomorphism.
This result can be further generalised to contractible curves on normal projective surfaces, see \cite[Lemma 4.3]{MZ22}.
Below is a beautiful result of Wahl \cite[Theorem 2.8]{Wah90}, which has been further generalized to higher dimensional cases, see Section \ref{sec-lcy}.

\begin{theorem}\label{thm-wah}
Let $X$ be a normal projective surface admitting a non-isomorphic surjective endomorphism. 
Then $X$ has at worst log canonical singularities.
In particular, the canonical divisor $K_X$ is $\Q$-Cartier.
\end{theorem}
 
Using the above result, we can show the following (see \cite[Theorem 4.7]{MZ22}).

\begin{theorem}\label{thm-emmp-sur}
	Let $X$ be a normal projective surface admitting a non-isomorphic surjective endomorphism. 
	Then any MMP starting from $X$ is $G$-equivariant for some finite-index submonoid $G$ of $\SEnd(X)$.
\end{theorem}

Based on the above theorem, the structural result \cite[Theorem 1.1]{JXZ23} can be given.  
We refer to \cite{Fuj12} for the Minimal Model Program of log canonical projective surfaces.

\begin{theorem}\label{thm-surf}
	Let $f : X \to X$ be a non-isomorphic surjective endomorphism
	of a normal projective surface.
	Then $X$ has only log canonical (lc) singularities.
	If the canonical divisor $K_X$ is pseudo-effective, then Case~(\ref{thm:main:qe}) below occurs.
	If $K_X$ is not pseudo-effective,
	replacing $f$ by an iteration,
	we may run an $f$-Equivariant Minimal Model Program (EMMP)
	\[
		X = X_1 \longrightarrow \cdots \longrightarrow X_j \longrightarrow \cdots \longrightarrow X_r \longrightarrow Y,
	\]
	contracting $K_{X_j}$-negative extremal rays,
	with $X_j \to X_{j+1}$ ($j < r$) being divisorial and $X_r \to Y$ being Fano contraction,
	such that one of the following cases occurs.
	\begin{enumerate}
		\item \label{thm:main:qe}
		      $f$ is quasi-{\'e}tale, i.e., {\'e}tale in codimension $1$;
		      there exists an $f$-equivariant quasi-{\'e}tale finite Galois cover
		      $\nu : V \to X$ with $V$ smooth as in \cite[Theorem 2.14]{JXZ23} and the lifting $f|_V$ of $f$ is \'etale.
		\item \label{thm:main:finite-ord}
		      $Y$ is a smooth projective curve of genus $g(Y) \geq 1$;
		      $f$ descends to an automorphism of finite order on the curve $Y$.
		\item \label{thm:main:smooth}
		      $Y$ is an elliptic curve;
		      $X \to Y$ is a $\mathbb{P}^1$-bundle.
		\item \label{thm:main:prdt}
		      $Y \cong \mathbb{P}^1$;
		      there is an $f|_{X_r}$-equivariant finite surjective morphism
		      $X_r \to Y \times \mathbb{P}^1$.
		\item \label{thm:main:cover}
		      $Y \cong \mathbb{P}^1$;
		      $f$ is polarized;
		      $K_X + S \sim_{\mathbb{Q}} 0$ for an $f^{-1}$-stable reduced divisor $S$;
		      there is an $f$-equivariant quasi-{\'e}tale finite Galois cover $\nu : V \to X$
		      where $V$ is a ruled surface over an elliptic curve and $\nu^* S$ is a disjoint union of two cross-sections (and the lifting of $f$ is \'etale outside $\nu^* S$).
		\item \label{thm:main:base-change}
		      $Y \cong \mathbb{P}^1$;
		      $f$ is polarized;
		      there exists an equivariant commutative diagram:
		      \[
			      \xymatrix{
			      \widetilde{f} \acts \widetilde{X} \ar@<2.3ex>[d]_{\widetilde{\pi}} \ar[r]^{\mu_{X}}	&	X \ar@<-2.3ex>[d]^{\pi} \racts f	\\
			      g_E \acts E \ar[r]												&	Y \racts g							\\
			      }
		      \]
		      here $E$ is an elliptic curve;
		      $\widetilde{X}$ is the normalisation of $X \times_Y E$;
		      $\widetilde{\pi}$ is a $\mathbb{P}^1$-bundle;
		      $\pi$ is a $\mathbb{P}^1$-fibration;
		      $\widetilde{f}$ and $g_E$ are finite surjective endomorphisms;
		      $\mu_X$ is quasi-{\'e}tale.
		\item \label{thm:main:fano-type}
		      $Y \cong \mathbb{P}^1$;
		      $X_r$ is of Fano type;
		      there is an $f|_{X_r}$-equivariant birational morphism $X_r \to \overline{X}$
		      to a klt Fano surface with Picard number $\rho(\overline{X}) = 1$;
		      every $X_j$ ($1 \leq j \leq r$) is a rational surface
		      whose singularities are klt (hence $\mathbb{Q}$-factorial).
		\item \label{thm:main:proj-cone}
		      $Y$ is a point;
		      $X_r$ is a projective cone over an elliptic curve $E$;
		      the normalisation $\Gamma$ of the graph of $X \dashrightarrow E$ is a $\mathbb{P}^1$-bundle,
		      and $f$ lifts to $f|_{\Gamma}$
		      such that $(f|_{\Gamma})^* |_{\N^1(\Gamma)} = \delta_f \id$.
		\item \label{thm:main:fano}
		      $Y$ is a point and hence $\rho(X_r)=1$;
		      $-K_{X_r}$ is ample;
		      every $X_j$ ($1 \leq j \leq r$) is a rational surface
		      whose singularities are lc and rational (hence $\mathbb{Q}$-factorial).
	\end{enumerate}
\end{theorem}

\begin{remark}
We refer to \cite{Nak02} for the precise list of birationally ruled surfaces which admit non-isomorphic surjective endomorphisms.
\end{remark}

\section{\'Etale EMMP for smooth projective threefolds}

Let $f:X\to X$ be a surjective endomorphism of a normal projective variety.
We have the ramification divisor formula
$$K_X=f^*K_X+R_f$$
where $R_f$ is the ramification divisor.
Recall that $f$ is quasi-\'etale or \'etale in codimension one if $R_f=0$.
Note that $K_X$ being pseudo-effective implies that $f$ is quasi-\'etale.
If $X$ is smooth and $f$ is quasi-\'etale, then $f$ is further \'etale by the purity of branch loci.

When the Kodaira dimension $\kappa(X)>0$, we can apply the equivariant Iitaka fibration \cite[Theorem A]{NZ09}.
\begin{theorem}
Let $f:X\dashrightarrow X$ be a dominant self-map of a smooth projective variety (or a compact K\"ahler manifold) $X$.
Then the Iitaka fibration $\phi:X\dashrightarrow Y$ is $f$-equivariant and $f|_Y$ is an automorphism of finite order.
\end{theorem}

{\bf In the following, we focus on a smooth projective threefold $X$.}

Suppose $\kappa(X)\ge 0$.
This is equivalent to $K_X$ being pseudo-effective by the abundance for threefolds (cf. \cite[Remark 3.13]{KM98}).
Moreover, $X$ has no Fano contraction and the first birational contraction, which exists when $K_X$ is not nef, 
has to be a divisorial contraction (cf. \cite[Theorem 3.3]{Mor82}).
In particular, the closed cone
$\NE(X)$ of effective $1$-cycles has only finitely many $K_X$-negative extremal rays (cf. \cite[Proposition 4.6]{Fuj02}).
With the help of non-isomorphic \'etale dynamics, Fujimoto \cite[Theorem 4.8]{Fuj02} further showed 
that the divisorial contraction has to be (the inverse of) the blowup along an elliptic curve.
In particular, one can run the following smooth EMMP.
 
\begin{theorem}
Let $f:X\to X$ be a non-isomorphic \'etale endomorphism of a smooth projective threefold $X$ with $\kappa(X)\ge 0$.
Then after iteration of $f$, any MMP starting from $X$ is $f$-equivariant and involves only divisorial contractions 
which are (the inverse of) the blowups along elliptic curves.
\end{theorem}

Furthermore, Nakayama and Fujimoto \cite[Main Theorem]{FN07} completely classified smooth projective threefolds 
which have non-negative Kodaira dimension and admit non-isomorphic surjective endomorphisms.

\begin{theorem}
Let $X$ be a smooth projective threefold with pseudo-effective $K_X$. 
Then the following conditions are equivalent to each other:
\begin{itemize}
\item[(A)] $X$ admits a non-isomorphic surjective endomorphism.
\item[(B)] There exist a finite \'etale Galois covering $\tau:\widetilde{X}\to X$ and 
an abelian scheme structure $\varphi:\widetilde{X}\to T$ over a variety $T$ of dimension $\le 2$ such that the
Galois group $\Gal(\tau)$ acts on $T$ and $\varphi$ is $\Gal(\tau)$-equivariant.
\end{itemize}
\end{theorem}

Now we consider the case when the Kodaira dimension $\kappa(X)<0$, or equivalently, $K_X$ is not pseudo-effective.
Of course, we can run MMP which ends with a Mori fibre space.
However, the MMP is no longer equivariant in general.
Here is an example.

\begin{example}
There exists a smooth rational surface $S$ admitting an automorphism $g$ of positive entropy with no $g$-periodic curves (cf.~\cite{Les21}).
Then we cannot run any $g$-EMMP.
Let $E$ be an elliptic curve and $n_E$ ($n \ge 2$) the multiplication map.
Take $X=S\times E$ and $f=g\times n_E$.
Clearly, $\kappa(X)<0$ and $f$ is \'etale.
Note that $X$ admits no $f^{-1}$-periodic prime divisor.
Indeed, let $P$ be an $f^{-1}$-invariant prime divisor and $e\in E$ the identity element.
Then $P\cap S\times \{e\}$ is $(f|_{S\times\{e\}})^{-1}$-invariant and hence $P=S\times \{a\}$ for some $a$. 
However, $E$ admits no $(n_E)^{-1}$-periodic point.
In particular, we cannot run any $f$-EMMP while $S$ and hence $X$ are clearly not relatively minimal.
\end{example}

Nevertheless, Fujimoto established the \'etale sequence of constant Picard number 
%(ESP for short) 
by running 
a not necessarily ``Equivariant'' (even after iteration) MMP.
We refer to \cite{Fuj18} and its sequels for the details.

\section{Quasi-amplified endomorphisms of klt log Calabi-Yau varieties}

A normal projective variety $X$ is of {\it Calabi-Yau type}, if $(X,\Delta)$ is log canonical and $K_X+\Delta\equiv 0$ 
for some effective Weil $\Q$-divisor $\Delta$.
The pair is called log Calabi-Yau.
Such a pair has $K_X+\Delta\sim_{\mathbb{Q}} 0$ ($\mathbb{Q}$-linear equivalence) by Theorem \ref{thm-gongyo}.
If the pair is further klt, we say $X$ is of klt Calabi-Yau type.

%Instead of EMMP, 
One can run the equivariant contractions for quasi-amplified endomorphism on $X$ of klt Calabi-Yau type as proved in \cite{Men23}.

\begin{theorem}\label{main-thm-qa-a}
Let $f:X\to X$ be a quasi-amplified endomorphism of a projective variety $X$ of klt Calabi-Yau type.
%Suppose $(X,\Delta)$ is klt and $K_X+\Delta\equiv 0$ for some effective Weil $\Q$-divisor $\Delta$.
Then replacing $f$ by a positive power, there is an $f$-equivariant sequence of birational contractions of extremal rays 
(in the sense of \cite[Definition 4.1]{MZ20a})
$$X=X_1\to \cdots \to X_i \to \cdots \to X_r$$
(i.e. $f=f_1$ descends to surjective endomorphism $f_i$ on each $X_i$), such that we have:
\begin{itemize}
\item[(1)] $f_r$ is amplified.
\item[(2)] For each $i$, $X_i$ is of klt Calabi-Yau type.
\item[(3)] For each $i$, $f_i$ is of positive entropy and quasi-amplified and $\Per(f_i)\cap U_i$ is countable and 
Zariski dense in $X_i$ for some open subset $U_i\subseteq X_i$. 
\item[(4)] Suppose the base field $k$ is uncountable. For each $i$, $f_i$ has a Zariski dense orbit.
\end{itemize}
\end{theorem}

The density of periodic points $\Per(f_i)$ is based on a famous result of Fakhruddin \cite{Fak03}.

\begin{theorem}\label{FakThm}
Let $f:X\to X$ be an amplified endomorphism of a projective variety over an algebraically closed field $k$ of arbitrary characteristic.
Then the set $\Per(f)$ of periodic points is Zariski dense in $X$.
\end{theorem}

As an application of Theorem \ref{main-thm-qa-a} to Hyperk\"ahler manifolds, we have:

\begin{theorem}\label{main-thm-hk}
Let $f:X\to X$ be an automorphism of a projective Hyperk\"ahler manifold $X$.
Then the following are equivalent.
\begin{itemize}
\item[(1)] $f$ is of positive entropy.
\item[(2)] $f^*D\not\equiv D$ for any nef $\mathbb{R}$-Cartier divisor $D\not\equiv 0$.
\item[(3)] $f$ is quasi-amplified.
\item[(4)] For some $n>0$, $f^n$ is birationally equivalent to some amplified automorphism $f':X'\to X'$.
\end{itemize}
Moreover, if $f$ is PCD, then all the above are satisfied.
\end{theorem}

\section{Fano threefolds with non-isomorphic endomorphisms}

In this section, we characterize smooth Fano threefolds admitting non-isomorphic surjective endomorphisms \cite[Main Theorem, Theorems 1.3 and 1.4]{MZZ22}.

\begin{theorem}
A smooth Fano threefold
is either toric or a product of $\mathbb{P}^1$ and a del Pezzo surface
if and only if
it admits a non-isomorphic surjective endomorphism.
\end{theorem}

To be precise, we have the following characterization. 
Note that a surjective endomorphism $f$ on a Fano variety or more generally on a Mori Dream Space, 
is amplified if and only if it is int-amplified (cf.~Proposition \ref{prop-iamp-amp}).

\begin{theorem}\label{main-thm-prod}
Let $X$ be a smooth Fano threefold. Then the following are equivalent.
\begin{enumerate}
\item[(1)]
$X$ is a product of $\mathbb{P}^1$ and a del Pezzo surface.
\item[(2)]
$X$ admits a non-isomorphic surjective endomorphism which is not polarized even after iteration.
\item[(3)]
$X$ admits a non-isomorphic surjective endomorphism which is non-amplified (or equivalently, non-int-amplified).
\end{enumerate}
\end{theorem}

\begin{theorem}\label{main-thm-toric}
Let $X$ be a smooth Fano threefold. Then the following are equivalent.
\begin{enumerate}
\item[(1)]
$X$ is a toric variety.
\item[(2)]
$X$ admits a polarized endomorphism.
\item[(3)]
$X$ admits an amplified
(or equivalently, int-amplified)
endomorphism.
\end{enumerate}
\end{theorem}

Theorem \ref{main-thm-toric} gives a partial answer to the following question, originally proposed by 
Fakhruddin \cite[Question 4.4]{Fak03} for polarized endomorphisms and later generalised to int-amplified cases \cite[Question 1.2]{MZg20}.
Question \ref{main-que-toric} itself generalizes a long-standing folklore Conjecture
of the 1980s, which focuses on Fano manifolds of Picard number $1$.

\begin{question}\label{main-que-toric}
Let $X$ be a rationally connected smooth projective variety. Suppose that $X$ admits an int-amplified (or polarized) endomorphism $f$. Is $X$ a toric variety?
\end{question}

Now we briefly explain the strategy of the proofs of the theorems in this section.

Recall that a smooth projective variety is {\it Fano} if $-K_X$ is ample.
By \cite[Corollary 1.3.2]{BCHM10}, a Fano manifold is a Mori Dream Space.
So $\NE(X)$ is a rational polyhedron (this can also be obtained by the cone theorem).
Then we can run EMMP for any surjective endomorphism $f:X\to X$.

The proof starts with the following Picard number $1$ case, due to Amerik, Rovinsky and Van de Ven \cite{Ame97, ARV99}, 
see also \cite{HM03} for a different approach.

\begin{theorem}
Let $X$ be a smooth Fano threefold of Picard number $1$ admitting a non-isomorphic surjective endomorphism.
Then $X\cong \mathbb{P}^3$.
\end{theorem}

The Picard number $\le 2$ case was settled in \cite[Theorem 5.1]{MZZ22}.

\begin{theorem}
Let $X$ be a smooth Fano threefold of Picard number $\rho(X)\le 2$ admitting a non-isomorphic surjective endomorphism $f$.
Then $X$ is toric. To be precise, $X$ is either $\mathbb{P}^3$, or a splitting $\mathbb{P}^1$-bundle over $\mathbb{P}^2$, 
or a blowup of $\mathbb{P}^3$ along a line.
\end{theorem}

For the case of Picard number $\ge 3$, we obtain a conic bundle structure \cite[Theorem 6.1]{MZZ22}.
Here a fibration $\tau:X\to Y$ of smooth projective varieties is a \textit{conic bundle} if every fibre is isomorphic to a conic, i.e., 
a scheme of zeros of a nonzero homogeneous form of degree $2$ on $\mathbb{P}^2$.
Denote by $\Delta_\tau:=\{y\in Y~|~\tau~\textup{is not smooth over }y\}$
the \textit{discriminant} locus of $\tau$.
When $Y$ is a surface, either $\Delta_\tau=\emptyset$ or $\Delta_\tau$ is a (not necessarily irreducible) curve 
with only ordinary double points (cf.~\cite[Section 6]{MM83}).
If $X$ is further assumed to be Fano, then we say $\tau$ is a \textit{Fano conic bundle}.
%With several technical results in \cite[Section 4]{MZZ22}, we can show that the conic bundle $\tau$ is a projective vector bundle.

\begin{theorem}
Let $X$ be a smooth Fano threefold with $\rho(X)\ge 3$.
Suppose that $X$ admits a non-isomorphic surjective endomorphism $f$.
Then $X$ admits a (Fano) conic bundle $\tau:X\to Y$.
%Further, $\tau$ is smooth and $X=\mathbb{P}_Y(\mathcal{E})$ for some rank $2$ vector bundle $\mathcal{E}$ on $Y$.
\end{theorem}

With the help of the Fano conic bundle structure, we can run the relative Minimal Model Program as discussed in \cite[Section 6]{MM83} 
and obtain the following special EMMP in the category of smooth Fano threefolds as stated in \cite[Theorem 6.2]{MZZ22}.

\begin{theorem}
Let $X$ be a smooth Fano threefold with a conic bundle $\tau:X\to Y$.
Suppose that $X$ admits a non-isomorphic surjective endomorphism $f$.
Then there exists an $f$-Equivariant (after iteration of $f$) Minimal Model Program (EMMP)
$$X=X_1\rightarrow \cdots \rightarrow X_i\rightarrow\cdots\rightarrow X_{r+1}\rightarrow Y$$
such that the following hold.

\begin{enumerate}
\item $r=\rho(X)-\rho(Y)-1$ and each $X_i$ is a smooth Fano threefold.
\item $\tau_{r+1}:X_{r+1}\to Y$ is a Fano contraction and an algebraic $\mathbb{P}^1$-bundle $\mathbb{P}_Y(\mathcal{E})$ over $Y$.
\item The composition $\tau_i:X_i\to Y$ is a conic bundle with
$(f|_Y)^{-1}$-invariant discriminant
$\Delta_{\tau_i}=C_i\cup\cdots \cup C_r$ a disjoint union of $r+1-i$ smooth curves on $Y$.
\item The composition $\pi:X\to X_{r+1}$ is the blowup of $X_{r+1}$ along $r$ disjoint union of
$(f|_{X_{r+1}})^{-1}$-invariant
smooth curves $\bigcup_{i=1}^r\overline{C_i}$ with $\tau_{r+1}(\overline{C_i})=C_i$.
\end{enumerate}
\end{theorem}

There remain two difficult issues:

\begin{itemize}
\item Confirming the splitting of $\mathcal{E}$ when $X=\mathbb{P}_Y(\mathcal{E})$ for some rank $2$ vector bundle $\mathcal{E}$ on $Y$, and
\item Verifying that each blowup in the EMMP is a toric blowup.
\end{itemize}

The splitting problem was dealt with in \cite[Theorems 6.4 and 7.1]{MZZ22} for the respectively polarized and non-polarized cases.
However, the EMMP in general is not a sequence of toric blowups. 
Indeed, one needs to choose very carefully a suitable EMMP to reach this goal as discussed in the proof of \cite[Theorems 8.1 and 7.1]{MZZ22} 
for the respectively int-amplified and non-int-amplified cases.

\begin{remark}\label{rmk-fano}
$ $
\begin{enumerate}
\item Theorem \ref{main-thm-toric} gives a partial answer (only for smooth Fano threefolds) to a question proposed by Fakhruddin \cite[Question 4.4]{Fak03}.
\item Jia and Zhong further developed the technique and proved that a smooth Fano fourfold 
admitting a conic bundle structure and an int-amplified endomorphism is indeed toric, see \cite{JZ20}.
\item Totaro classified smooth Fano 3-folds satisfying the Bott vanishing \cite[Theorem 0.1]{Tot23a}. 
Kawakami and Totaro proved that every normal projective variety admitting 
an int-amplified endomorphism $f$ satisfies the Bott vanishing (cf.~Theorem \ref{thm-KT}).
This way, Totaro \cite{Tot23b} obtained a different proof of Theorem \ref{main-thm-toric}, which
works for the known list of smooth Fano threefolds defined over an algebraically closed field $k$ with $\deg f$ invertible in $k$.
\end{enumerate}

\end{remark}

\section{Application to Broustet-Gongyo conjecture}\label{sec-lcy}

A generalisation of Question \ref{main-que-toric} is the following conjecture proposed by 
Broustet and Gongyo \cite{BG17}. They confirmed the conjecture for surfaces and Mori Dream Spaces, using the EMMP.
Recall that a normal projective variety $X$ is of {\it Calabi-Yau type} 
if $(X,\Delta)$ is an lc pair for some effective Weil $\Q$-divisor $\Delta$ such that $K_X+\Delta\sim_{\Q} 0$.
By the abundance (cf.~Theorem \ref{thm-gongyo}), the latter condition is equivalent to $K_X+\Delta\equiv 0$.
We also call such a pair $(X, \Delta)$ {\it log Calabi-Yau}.

\begin{conjecture}\label{main-conj-cy}
Let $X$ be a normal projective variety admitting a polarized endomorphism.
Then $X$ is of Calabi-Yau type.
\end{conjecture}

For the $X$ in the above conjecture, assuming further that the canonical divisor $K_X$ is $\Q$-Cartier,
Broustet and H\"oring \cite{BH14} proved that $X$ has at worst log canonical singularities; see also Theorem \ref{thm-wah}.
Conjecture \ref{main-conj-cy} has been verified for rationally connected smooth projective varieties 
by Yoshikawa since Fano type is Calabi-Yau type; see \cite{Yos21} or Theorem \ref{thm-yoshikawa}.
Based on this, the following theorem in \cite{Men22} fully solved the case of smooth projective threefolds.

\begin{theorem}\label{main-thm-lcy}
Let $X$ be a smooth projective threefold admitting a polarized endomorphism $f$.
Then $X$ is of Calabi-Yau type.
\end{theorem}

We briefly explain the strategy of the proof of Theorem \ref{main-thm-lcy}.
We first run the $f$-Equivariant Minimal Model Program which ends with a $Q$-abelian variety $Y$, i.e., 
a quasi-\'etale (or \'etale in codimension one) quotient of an abelian variety, see Theorem \ref{main-thm-GMMP}.
Since the smooth rationally connected case has been verified by Yoshikawa, one may assume $\dim Y>0$.
Then we study the fibration $\pi:X\to Y$ and its $f$-periodic general fibres.
By applying the canonical bundle formula and ramification divisor formula, the natural idea is 
to reduce the problem to the following Conjecture \ref{main-conj-lc} (for surfaces) proposed by Gongyo.

\begin{conjecture}[Gongyo]\label{main-conj-lc}
Let $f:X\to X$ be a $q$-polarized endomorphism of a smooth projective variety $X$.
Then $(X,\frac{R_f}{q-1})$ is an lc pair after iteration of $f$.
\end{conjecture}

However, in dimension two, Conjecture \ref{main-conj-lc} is not known even for $\mathbb{P}^2$, which is, in fact, the only leftover case.
So the main difficulty of proving Theorem \ref{main-thm-lcy} remains in the case where $\pi:X\to Y$ is a $\mathbb{P}^2$-bundle over an elliptic curve $Y$.
By applying the Iitaka fibration of the anti-canonical divisor, we are reduced to the case 
where the anti-canonical divisor $-K_X$ is big (i.e., it is the sum of an ample divisor and an effective divisor), see \cite[Section 5]{Men22}.
Then the problem boils down to the very concrete case: $X\cong\mathbb{P}_Y(\mathcal{F}_2\oplus \mathcal{L})$ 
where $\mathcal{F}_2$ is the unique indecomposable rank 2 vector bundle of degree $0$
with non-trivial global sections and $\mathcal{L}$ is a line bundle of negative degree.

\section{Application to Kawaguchi-Silverman conjecture}

In this section, we work over a number field $K$ with a fixed algebraic closure $\overline{K}$.
%$\overline{\Q}$..
Let $f: X \to X$ be a surjective endomorphism of a normal projective variety $X$.
One may consider two dynamical invariants: the (first) {\it dynamical degree} 
$$\delta_f= \lim_{n \to +\infty} ((f^n)^*H \cdot H^{\dim (X) -1})^{1/n}$$
of $f$ and the {\it arithmetic degree} 
$$\alpha_f(x)=\lim_{n \to +\infty} h_H(f^n(x))^{1/n}$$
of $f$ at $x \in X(\overline K)$.
Here $H$ is any ample divisor on $X$ and $h_{H}$ is (the maximum of $1$ and) the Weil height function associated with $H$.
We refer to \cite[Part B]{HS00}, \cite{KS16b}, or \cite[Section 2.2]{Mat20} for the detailed definition of the {\it Weil height function}.
We refer to \cite[Section 2]{MMSZ23} for some basic properties of these invariants. 
%\changemade{Comment 3. I added definitions of the two invariants.}

The dynamical degree reflects the geometric complexity of iterations of $f$.
On the other hand, the arithmetic degree reflects the arithmetic complexity of the given (forward) $f$-{\it orbit}
$O_f(x) :=\{ x, f(x), f^2(x), \ldots \}$.

One naturally likes to compare these two invariants.
By Kawaguchi--Silverman \cite{KS16a}, the arithmetic degree at any point is less than or equal to the dynamical degree (see also \cite{Mat20}).
Then, one wonders when the arithmetic degree at a point attains (its upper bound) the dynamical degree.
Kawaguchi and Silverman proposed (see \cite{KS16a}):

\begin{conjecture}\label{conj_ks} ( {\bf Kawaguchi-Silverman Conjecture, KSC for short})
Let $f: X \to X$ be a surjective endomorphism of a projective variety $X$ defined over a number field $K$.
Let $x \in X(\overline K)$ such that $ \alpha_{f}(x)< \delta_{f}$.
Then the orbit $O_f(x)$ is not Zariski dense in
$X_{\overline{K}} := X \times_K \overline{K}$.
\end{conjecture}

\begin{remark}
$ $
\begin{enumerate}
\item The original conjecture in \cite{KS16a} has also been formulated for rational maps.
For endomorphisms, Kawaguchi and Silverman proved that $\alpha_{f}(x)$ equals either $1$ or the modulus of an eigenvalue of $f^*|_{\NS(X)}$. 
For rational maps, the existence of the limit defining $\alpha_{f}(x)$ is the first part of the original conjecture, and Matsuzawa \cite[Theorem 1.4]{Mat20} proved that the supremum limit still satisfies 
$$\overline{\alpha}_{f}(x):=\limsup\limits_{n \to +\infty} h_H(f^n(x))^{1/n}\le \delta_f.$$
\item The KSC was further generalised to the (stronger) small Arithmetic Non-Density conjecture (sAND for short) in \cite{MMSZ23}, 
see \cite[Conjecture 1.3 and Remark 1.4]{MMSZ23}.
The sAND conjecture implies the KSC and is much harder than the KSC with obstacles proposed as three questions in \cite[Section 8]{MMSZ23}.
We refer to \cite{BMS21,MMSZZ20} for recent progress towards these questions.
\end{enumerate}

\end{remark}

We now summarise the known results on KSC till the very recent, to the best of our knowledge.

\begin{theorem}\label{thm-ksc}
The KSC holds true for a surjective endomorphism $f$ 
on a projective variety $X$ which fits one of the following cases.
\begin{enumerate}
\item $f$ is polarized (cf. \cite{KS14}).
\item $X$ is a Mori dream space (cf. \cite{Mat20a}).
\item $\dim(X)\le 2$ (for smooth case, see \cite{MSS18}; for singular case, see \cite{MZ22}).
\item $X$ is a Hyperk\"ahler manifold (cf. \cite{LS21}).
\item $X$ is a ($Q$)-abelian variety (cf. \cite{Sil17}).
\item $X$ is a smooth rationally connected projective variety admitting an int-amplified endomorphism (see \cite{MZ22} and \cite{MY22} for respectively threefolds and $n$-folds).
\item $X$ is a projective threefold which admits an int-amplified endomorphism and has at worst $\Q$-factorial terminal singularities (cf. \cite[Corollary 6.19]{MMSZ23}).
\item $X$ is a smooth projective threefold with $q(X)>0$ and $f$ is an automorphism (cf. \cite{CLO22}).
\end{enumerate}
\end{theorem}

\begin{remark}
In fact, with \cite[Theorems 1.5 and 3.1]{MMSZ23} we confirm the stronger sAND \cite[Conjecture 1.3]{MMSZ23} and hence the KSC 
for Cases (1) - (7) in Theorem \ref{thm-ksc}, see \cite[Remark 1.4]{MMSZ23}.
However, the sAND remains unsolved for Case (8) in Theorem \ref{thm-ksc}.
\end{remark}

\begin{remark}
When $f$ is only a dominant self-map, KSC has been verified at least for the following cases.
\begin{enumerate}
\item $f$ is certain self-map of $\mathbb{P}^n$ (cf. \cite[Theorem 2]{KS14}).
\item $f$ is a non-isomorphic finite surjective endomorphism of a smooth affine surface of non-negative Kodaira dimension (cf. \cite{JSXZ21}).
\item $f$ is birational and $X$ is a smooth projective variety with Kodaira dimension $0$ and irregularity $q(X)\ge \dim X-1$ (cf. \cite{CLO22}).
\item $X$ is a smooth projective variety with Kodaira dimension $0$ and $q(X)=\dim X$ (cf. \cite{CLO22}). 
\item A slightly weaker version of KSC (for generic orbits) is proved when $f$ is $1$-cohomologically hyperbolic, 
i.e., the first dynamical degree is larger than all the other dynamical degrees (cf. \cite{MW22}).
\end{enumerate}
\end{remark}

We now briefly explain the EMMP used in proving Theorem \ref{thm-ksc} (3), (6) and (7).

When there is an int-amplified endomorphism, we are able to reduce KSC to the highly geometrically restrictive 
Case TIR by the following theorem (cf. \cite[Theorem 1.7]{MZ22}).
The reduction is based on the effectivity of $-K_X$ (cf.~\cite[Theorem 1.5]{MZ22}), the equivariant 
Chow reduction (cf.~\cite[Proposition 1.6]{MZ22}), and a carefully chosen EMMP (cf.~\cite[Propostion 9.2]{MZ22}).

\begin{theorem}\label{main-thm-tir}
Let $X$ be a normal projective variety having only $\Q$-factorial Kawamata log terminal (klt)  singularities and one int-amplified endomorphism.
Then we have:
\begin{itemize}
\item[(1)] If $K_X$ is pseudo-effective, then KSC holds for any surjective endomorphism of $X$.
\item[(2)] Suppose that KSC holds for Case TIR
(for those $f|_{X_i} : X_i \to X_i$ appearing in any EMMP starting from $X$) to be defined below.
Then KSC holds for any (not necessarily int-amplified) surjective endomorphism $f$ of $X$.
\end{itemize}
\end{theorem}

{\bf Case TIR}$_n$ (Totally Invariant Ramification case).
Let $X$ be a normal projective variety of dimension $n \ge 1$, which has only $\Q$-factorial 
Kawamata log terminal (klt) singularities and admits one int-amplified endomorphism.
Let $f:X\to X$ be an arbitrary surjective endomorphism.
Moreover, we impose the following conditions.
\begin{itemize}
\item[(A1)]
The anti-Kodaira dimension $\kappa(X,-K_X)=0$; $-K_X$ is nef, whose class is extremal in both the {\it nef cone} 
$\Nef(X)$ and the {\it pseudo-effective divisors cone} $\PE^1(X)$.
\item[(A2)]
$f^*D = \delta_f D$ for some prime divisor $D\sim_{\Q} -K_X$.
\item[(A3)]
The ramification divisor of $f$ satisfies $\Supp \, R_f = D$.
\item[(A4)]
There is an $f$-equivariant Fano contraction $\pi:X\to Y$ with $\delta_f>\delta_{f|_Y}$ ($\ge 1$).
\end{itemize}

Indeed, we conjecture that Case TIR will never happen, which is confirmed for rationally connected smooth varieties (see \cite{MZ22} and \cite{MY22} for threefolds
and $n$-folds respectively) and is further evidenced by \cite[Theorem 6.6]{MMSZ23}, i.e., the following theorem, while the latter in turn, 
implies Theorem \ref{thm-ksc} (3), (6), and (7).
For the surface case, we may remove the assumption on singularities in Case TIR (cf.~\cite[Theorem 5.2]{MZ22}), see also Theorem \ref{thm-surf}.

\begin{theorem}
Suppose that the $X$ in Case TIR has only ($\Q$-factorial) terminal singularities.
Then $\dim(X)$ $\ge \dim(Y)+3\ge 4$.
\end{theorem}

For Theorem \ref{thm-ksc} (8), the essential case is when the Kodaira dimension $\kappa(X)=-\infty$, our $f$ is an automorphism of positive entropy, 
and the Albanese map $\alb:X\to B$ is a surjective morphism to an elliptic curve $B$.
Chen, Lin and Oguiso applied the classification result of Lesieutre \cite[Theorem 1.7]{Les18} to run smooth EMMP.
This way, one can further narrow down to the case where $X$ admits an $f$-equivariant conic bundle structure with the first dynamical degree preserved, 
thus reducing to the surface case, where KSC is known, see Theorem \ref{thm-ksc} (3).

\section{Application to Zariski Dense Orbit conjecture}

In this section, we work over an algebraically closed field $\bk$ of characteristic $0$.
We consider the following conjecture.
Note that Conjecture \ref{conj:ZD} (\ref{conj:ZD:extend}) below is invariant under birational conjugation
which is stronger than the long-standing Conjecture \ref{conj:ZD} (\ref{conj:ZD:classical}), called the Zariski dense orbit conjecture.
In fact, they are equivalent
modulo the Dynamical Mordell-Lang Conjecture, see \cite[Proposition 2.6]{Xie22}.
Recall that a rational function $\psi \in \bk(X)$ is \emph{$f$-invariant}
if $f^*(\psi) := \psi \circ f = \psi$.
Denote by $\bk(X)^f$ the field of $f$-invariant rational functions on $X$.
We have $\bk \subseteq \bk(X)^f \subseteq \bk(X)$.
Denote by $X(\bk)_f$ the set of points whose $f$-orbit $O_f(x)$ is well-defined.

\begin{conjecture}\label{conj:ZD} 
	Let $X$ be an (irreducible) projective variety over $\bk$
	and $f : X \dashrightarrow X$ a dominant rational self-map such that $\bk(X)^f = \bk$.
	Then:
	\begin{enumerate}
		\item \label{conj:ZD:classical}
		      There is a point $x \in X(\bk)_f$ such that its $f$-orbit is Zariski-dense in $X$.
		\item \label{conj:ZD:extend}
		      For every Zariski-dense open subset $U$ of $X$,
		      there exists a point $x\in X(\bk)_f$ whose orbit $O_f(x)$ is contained in $U$ and Zariski-dense in $X$.
		      In particular, the set 
	$$\{ x \in X(k)_f : O_f(x) \text{ is Zariski-dense in } X \}$$
is Zariski dense in $X$. 
	\end{enumerate}
\end{conjecture}

We can extend Conjecture \ref{conj:ZD} (Zariski-topology version)
to a stronger Conjecture \ref{conj:AZD} below (adelic-topology version).
The adelic topology on $X(\bk)$ was proposed by Xie \cite[\S 3]{Xie22} with the following properties.

\begin{enumerate}
	\item It is stronger than the Zariski topology.
	\item It is ${\mathsf{T}}_1$,
		      i.e., for any two distinct points $x, y \in X(\bk)$
		      there are adelic open subsets $U, V$ of $X(\bk)$
		      such that $x \in U, y \notin U$ and $y\in V, x \notin V$.
	\item Morphisms between algebraic varieties over $\bk$
		      are continuous for the adelic topology.
		\item Flat morphisms are open with respect to the adelic topology.
		\item \label{ppty-of-adelic-topo:irreducible}
		      The irreducible components of $X(\bk)$ in the Zariski topology
		      are the irreducible components of $X(\bk)$ in the adelic topology.
		\item Let $K$ be any subfield of $\bk$ which is finitely generated over $\mathbb{Q}$
		      and such that $X$ is defined over $K$ and $\overline{K} = \bk$.
		      Then the action
		      \[
			      \Gal(\bk/K)\times X(\bk)\to X(\bk)
		      \]
		      is continuous with respect to the adelic topology.
\end{enumerate}
When $X$ is irreducible, (\ref{ppty-of-adelic-topo:irreducible}) implies that the intersection of finitely many non-empty adelic open subsets of $X(\bk)$ is non-empty.
So, if $\dim X\geq 1$, the adelic topology is not Hausdorff.
In general, the adelic topology is strictly stronger than the Zariski topology.

The adelic version of the Zariski-dense orbit conjecture was proposed in \cite{Xie22}.

\begin{conjecture}\label{conj:AZD}
	Assume that the transcendence degree of $\bk$ over $\mathbb{Q}$ is finite.
	Let $X$ be an irreducible variety over $\bk$ and
	$f : X \dashrightarrow X$ a dominant rational map.
	If $\bk(X)^f= \bk$, then there exists a non-empty adelic open subset $A \subseteq X(\bk)$
	such that the orbit of every point $x \in A\cap X(\bk)_f$ is Zariski-dense in $X$.
\end{conjecture}

\begin{definition}
	Let $X$ be an (irreducible) projective variety over $\bk$
	and $f : X \dashrightarrow X$ a dominant rational map.
	We say that a pair $(X,f)$ satisfies
	\begin{enumerate}
		\item \emph{ZD-property}, if Conjecture \ref{conj:ZD}(\ref{conj:ZD:classical}) holds true;
		\item \emph{Strong ZD-property}, if Conjecture \ref{conj:ZD}(\ref{conj:ZD:extend}) holds true;
		\item \emph{AZD-property}, if Conjecture \ref{conj:AZD} holds true; and
		\item \emph{SAZD-property}, if there is a non-empty adelic open subset $A$ of $X(\bk)$
		      such that for every point $x\in A\cap X(\bk)_f$, its orbit $O_f(x)$ is Zariski-dense in $X$.
	\end{enumerate}
\end{definition}

\begin{remark}\label{rem:AZD_to_ZD}
	%\hfill
	Conjecture \ref{conj:AZD} implies Conjecture \ref{conj:ZD}.
	Precisely, we have:
	\begin{enumerate}
		\item SAZD-property implies AZD-property.
		\item Conjecture \ref{conj:AZD} (adelic-topology version) is stronger
		      than the classical Conjecture \ref{conj:ZD} (Zariski-topology version).
		      Indeed, even the hypothesis on $\bk$ in Conjecture \ref{conj:AZD} does not cause any problem.
		      To be precise,
		      for every pair $(X,f)$ over $\bk$, there exists an algebraically closed subfield $K$
		      of $\bk$ whose transcendence degree over $\mathbb{Q}$ is finite
		      and such that $(X,f)$ is defined over $K$,
		      i.e., there exists a pair $(X_K,f_K)$ such that $(X,f)$ is its base change by $\bk$.
		      By \cite[Corollary 3.29]{Xie22},
		      if $(X_K,f_K)$ satisfies AZD-property, then $(X,f)$ satisfies strong ZD-property.
	\end{enumerate}
\end{remark}

The EMMP (cf. Theorem \ref{thm-surf}) can be applied to prove Conjecture \ref{conj:AZD} 
(and hence Conjecture \ref{conj:ZD}) for surjective endomorphisms of (possibly singular) projective surfaces, extending the smooth case in \cite{Xie22}.

\begin{theorem}\label{thm:AZD}
	Let $f : X \to X$ be a surjective endomorphism of a projective surface $X$ defined over the algebraically closed field $\bk$.
	Assume that $\bk$ has finite transcendence degree over $\mathbb{Q}$ (see also Remark \ref{rem:AZD_to_ZD} (2)).
	Then Conjecture \ref{conj:AZD} holds for $(X, f)$.
	Precisely, either $\bk(X)^f \neq \bk$,
	or there is a non-empty adelic open subset $A \subseteq X(\bk)$ such that the forward orbit $O_f(x)$ of every point $x \in A$ is Zariski-dense in $X$.
\end{theorem}

\begin{remark}
Even the case of normal projective surfaces of Picard number $1$ is highly non-trivial, see \cite[Thoerem 1.12]{JXZ23}
\end{remark}

We end this section by mentioning known results for Conjecture \ref{conj:ZD} (\ref{conj:ZD:classical}).
\begin{remark}
$ $
\begin{enumerate}
\item Conjecture \ref{conj:ZD} (\ref{conj:ZD:classical}) was proved by Amerik and Campana \cite[Theorem 4.1]{AC08}
	when the field $\bk$ is uncountable.
\item  Xie \cite[Appendix B, joint work with Thomas Tucker]{Xie22} proved Conjecture \ref{conj:AZD} and hence Conjecture \ref{conj:ZD} are proved
		      for 
		      $$f = (f_1,\cdots,f_n) \colon (\mathbb{P}^1)^n \to (\mathbb{P}^1)^n,$$
		      where the $f_i$'s are endomorphisms of $\mathbb{P}^1$;
		      see also \cite[Theorem 14.3.4.2]{BGT16},
		      when $f_i$'s are not post-critically finite,
		      and Medvedev and Scanlon \cite[Theorem 7.16]{MS14}
		      for an endomorphism 
		      $f \colon \mathbb{A}^n \to \mathbb{A}^n$
		      with $$f(x_1,\cdots,x_n) = (f_1(x_1),\cdots,f_n(x_n)),\, f_i(x_i) \in \bk[x_i].$$
		\item Xie \cite[Theorem 1.1]{Xie17} proved
		      Conjecture \ref{conj:ZD}(\ref{conj:ZD:classical}) for dominant polynomial endomorphism
		      $f \colon \mathbb{A}^2 \to \mathbb{A}^2$.
		\item Ghioca and Scanlon proved Conjecture \ref{conj:ZD}(\ref{conj:ZD:classical}) when $X$ is 
		an abelian variety and $f$ is a dominant self-map, see \cite[Theorem 1.2]{GS17}. The result is further generalized to the case when $X$ is semi-abelian by Ghioca and Satriano, see \cite[Theorem 1.1]{GS19};
		      Xie proved Conjecture \ref{conj:AZD} in this abelian variety case, 
		      see \cite[Theorem 1.14]{Xie22}.
		\item When $X$ is an algebraic surface and $f$ is a birational self-map, Xie \cite[Corollary 3.31]{Xie22} proved
		      Conjecture \ref{conj:AZD} (and hence Conjecture \ref{conj:ZD});
		      see also \cite[Theorem 1.3]{BGT14} when $X$ is quasi-projective
		      over $\overline{\mathbb{Q}}$ and $f$ is an automorphism.
		      \item
		      Conjecture \ref{conj:AZD} and hence Conjecture \ref{conj:ZD} have been proved
		      when $X$ is a \emph{smooth} projective surface
		      and $f$ is a surjective endomorphism in \cite[Theorem 1.15]{Xie22}.
\end{enumerate}
\end{remark}

\end{document}